\documentclass{article}

\usepackage{amsmath,amsfonts,amssymb,amsthm}

\newtheorem{thm}{Theorem}[section]
\newtheorem{prop}[thm]{Proposition}

\newtheorem{lem}[thm]{Lemma}
\newtheorem*{thm1}{Theorem 1}
\newtheorem*{thm2}{Theorem 2}

\newcommand{\bF}{\mathbf{F}}

\begin{document}

\title{The space of two-dimensional vectors over a four-dimensional division algebra over $\bF_2$}

\author{Daisuke Tambara\\
Hirosaki University}

\date{}

\maketitle

\begin{abstract}
 Let $A$ be a non-associative division algebra over a field $F$. Let $A$ act on the space $A^2$ by left multiplication. For nonzero elements $v, v'$ of $A^2$ we ask when the subspaces $Av$ and $Av'$ coincide. The paper gives an answer in the case where $A$ is four-dimensional and $F$ is the field of order 2.
In this case $A$ is known to be isotopic to either of two algebras,  system $V$ and system $W$ of Knuth. 
For each of these two algebras we give an explicit solution of the equation $Av=Av'$. The result is stated for any four-dimensional algebra $A$ defined over an arbitrary base field $F$ and equipped with the  multiplication rule of system $V$ or system $W$.

\end{abstract}

\section*{Introduction}

\medskip
Let $A$ be a division algebra  over a field $F$. 
Let $A$ act on the space $A^2=A\oplus A$ by left multiplication: 
$a(x,y)=(ax,ay)$. When $A$ is associative, $A^2$ is just a vector space over $A$.
But when $A$ is not associative,  as noted by Albert \cite{Alb},  basic facts about vector spaces do not hold for $A^2$.
This motivates us to study  properties of the space $A^2$ as contrasted with ordinary vector spaces.
Here considered is a question of proportionality of vectors.
An element $v \in A^2$ generates the subspace $Av=\{av\mid a\in A\}$ of $A^2$.
Our  question is when the equality $Av=Av'$ holds for $v,v'\in A^2$.
If $A$ is associative, this has a trivial answer.
In our previous work \cite{Tam} we dealt with the question in the case where $F$ is a finite field and $A$ is a three-dimensional non-associative division algebra to obtain the answer:
If the two components of $v$ are linearly independent over $F$, then $Av=Av'$ only when $Fv=Fv'$. 

In this paper we deal with the question in the next smallest case, that is, the case where $F=\bF_2$ the field of order 2 and $A$ is four-dimensional over $F$.
Such algebras were classified by Kleinfeld (\cite{Klein}). According to  Knuth (\cite{Knu}) any four-dimensional non-associative division algebra over $F=\bF_2$ is isotopic to either of the following two algebras.
Let $E=\bF_4$ the field of order 4 and let $\sigma$ be the nontrivial automorphism of $\bF_4$, that is, $a^{\sigma}=a^2$. 

\noindent
(1)  $A=E\oplus E=\{(a,b)\mid a,b\in E\}$ with multiplication 
$$
(a,b)(c,d)=(ac+b^{\sigma}d, bc+a^{\sigma}d+b^{\sigma}d^{\sigma}).
$$

\noindent
(2)  $A=E\oplus E$ with multiplication 
$$
(a,b)(c,d)=(ac+\omega b^{\sigma}d, bc+a^{\sigma}d),
$$
where $\omega\in E-F$ is a fixed element.

Knuth named (1) system $V$ and  (2)  system $W$.

The equation $Av=Av'$ being invariant under isotopy, 
we may restrict the algebra $A$ to  
 these representative algebras.
We then have the following solution of the equation $Av=Av'$ respectively.

\begin{thm1} Let $A$ be the algebra of (1).
Let $x,y,x',y'\in A$ be nonzero.
Write $x=(a,b), y=(c,d), x'=(a',b'),y'=(c',d')$
 with $a,b,\dots,c',d'\in E$.
Then $A(x,y)=A(x',y')$ if and only if one of the following holds:

(i) $x=y, x'=y'$.

(ii) $x=x', y=y'$.

(iii) $b=d=0$, $b'=d'=0$, $c=k a$, $c'=k a'$ for some $k\in E^{\times}$.

(iv) There exist $l,m,k\in E^{\times}$ such that
\begin{alignat*}{4}
a^{\sigma}&=l b, \quad& b^{\sigma}&=l a, \quad&
a^{\prime\sigma}&=l b',\quad& b^{\prime\sigma}&=l a',\\
c^{\sigma}&=m d, & d^{\sigma}&= mc, \quad&
c^{\prime\sigma}&=m d', & d^{\prime\sigma}&=m c',\\
c&=t a, & d&=t b, \quad&
c' &=t a', & d'&=t b'.
\end{alignat*}

\end{thm1}

\begin{thm2} 
Let $A$ be the algebra of (2).
Regard $A$ as a two-sided vector space over $E$ by left multiplication
$k(a,b)=(ka,k^{\sigma}b)$
and right multiplication
$(a,b)k=(ak,bk)$.
Let $x,y,x',y'\in A$ be nonzero elements. 
Then $A(x,y)=A(x',y')$ if and only if either of the following holds:

(i) $x'=k x$, $y'=k y$ for some $k\in E^{\times}$.

(ii) $x=ym$, $x'=y'm$ for some $m\in E^{\times}$.

\end{thm2}

\medskip
Actually we have similar results over a general base field.
Let $F$ be a field of arbitrary characteristic and $E/F$ a quadratic extension with nontrivial automorphism $\sigma$. Let $\omega\in E-F$.
Then $A=E\oplus E$ becomes an algebra over $F$ with multiplication given by the formula of (1) or (2).
We call $A$ the algebra of type V or type W respectively.
The algebra of type W is always a division algebra, while
the algebra of type V is a division algebra if and only if $t^2-3t+1$ is irreducible over $F$.
In each of these cases we give the solution for the equation $Av=Av'$. 
(Theorem 2.1, Theorem  2.2.)

For the algebra of type V we follow Kaplansky's  algebraically closed style (\cite{Kap}):
By extension of scalars we pass to an algebra of split type, where we make direct computations. For the algebra of type W we can readily solve the equation by virtue of the existence of the two-dimensional nucleus.

The paper is organized as follows.
In Section 1 we make basic definitions and introduce the algebras of type V and type W.
In Section 2 we state the main results on the solution of $A(x,y)=A(x',y')$ for these algebras $A$.
In Section 3 we introduce the split version of the algebra of type V.
In Section 4 we solve the equation $A(x,y)=A(x',y')$ for the split algebra $A$ of  type V under some restriction on $x,y,x',y'\in A$.
From this 
we deduce in Section 5 the main result for type V.
In Section 6 we prove the result for type W.

\section{Preliminaries}

Let $F$  be a base field. 
An algebra is a vector space equipped with bilinear multiplication.
For any element $a$ of an algebra $A$ let $L_a$ and $R_a$ denote respectively the left and right multiplication by $a$, considered as a linear map $A\to A$: $L_a(x)=ax$, $R_a(x)= xa$.
An algebra $A$ is called a division algebra if $L_a$ and $R_a$ are bijective for every nonzero $a\in A$.

Two algebras $A$ and $B$ are said to be isotopic if there exists a triple of  linear isomorphisms
$f,g,h\colon A\to B$ such that $f(x)g(y)=h(xy)$ for all $x,y\in A$.

Let $A^2=A\oplus A=\{(x,y)\mid x,y\in A\}$.
Let $A$ act on $A^2$ by left multiplication
$a(x,y)=(ax,ay)$.
An element $v\in A^2$ generates the subspace $Av=\{av\mid a\in A\}$ of $A^2$.

The following proposition translates our equation $Av=Av'$ into an equation for linear transformations on $A$.

\begin{prop}
Let $A$ be any algebra.
Let $x,y,x',y'\in A$. Suppose that $R_{x},R_{y},R_{x'}, R_{y'}$ are invertible.
Then $A(x,y)=A(x',y')$ if and only if 
$R_{x'}^{-1}R_{x}=R_{y'}^{-1}R_{y}$, or equivalently
$R_{x}R_{y}^{-1}=R_{x'}R_{y'}^{-1}$.

\end{prop}

\medskip
\begin{proof}
Suppose $A(x,y)=A(x',y')$. For any $a\in A$ there exists $a'\in A$ such that
$ax=a'x',\;ay=a'y'$.
Then
$a'=R_{x'}^{-1}R_{x}(a), \;
a'=R_{y'}^{-1}R_{y}(a)$,
hence
$R_{x'}^{-1}R_{x}(a)=R_{y'}^{-1}R_{y}(a)$.
Thus $R_{x'}^{-1}R_{x}=R_{y'}^{-1}R_{y}$.

Conversely suppose $R_{x'}^{-1}R_{x}=R_{y'}^{-1}R_{y}$.
Reversing the above argument, we see $A(x,y)\subset A(x',y')$.
We have also $R_{x}^{-1}R_{x'}=R_{y}^{-1}R_{y'}$,  therefore
$A(x',y')\subset A(x,y)$.
Thus $A(x,y)=A(x',y')$.
\end{proof}

If a triple of linear isomorphisms $f,g,h\colon A\to B$ gives an isotopy,
the equation $A(x,y)=A(x',y')$ is equivalent to the equation
$B(g(x),g(y))=B(g(x'), g(y'))$.

Now we define two classes of algebras as
obvious generalizations of system $V$ and system $W$ of Knuth.  
Let $E/F$ be a quadratic field extension with nontrivial automorphism $\sigma$.
Let $N\colon E\to F$ be the norm map: $N(x)=xx^{\sigma}$.

Let $A=E\oplus E$. This is  a four-dimensional space over $F$.
Define multiplication on $A$ by either of the following rule. 

\noindent
(1)  
$(a,b)(c,d)=(ac+b^{\sigma}d, bc+a^{\sigma}d+b^{\sigma}d^{\sigma})$.

\noindent
(2)  
$(a,b)(c,d)=(ac+\omega b^{\sigma}d, bc+a^{\sigma}d)$,
where $\omega$ is a fixed element of $E-F$.

We call the resulting algebra $A$ respectively the algebra of type V or type W associated with the extension $E/F$.
When $F=\bF_2$, these specialize to the algebras in Introduction.

For the algebra of type V one sees that the left multiplication $L_{(a,b)}$ and right multiplication $R_{(a,b)}$ have determinant
$N(a)^2-3N(a)N(b)+N(b)^2$.
Therefore, $A$ is a division algebra if the quadratic $t^2-3t+1$ is irreducible over $F$, which is the case for $F=\bF_2$. 

For the algebra of type W
the left multiplication $L_{(a,b)}$ and right multiplication $R_{(a,b)}$ have determinant
$(N(a)-\omega N(b))(N(a)-\omega^{\sigma}N(b))$.
Therefore $A$ is always a division algebra.

In either type  
the subspace $A_0=E\oplus 0=\{(a,0)\mid a\in E\}$ is a subalgebra.
In case of type W, as mentioned in \cite{Knu},  $A_0$ is the nucleus of $A$, meaning that the associativity
$(xy)z=x(yz)$ holds whenever one of $x,y,z$ belongs to $A_0$.
In case of type V $A_0$ is the weak nucleus in the sense of Knuth (\cite{Knu}), meaning that
$(xy)z=x(yz)$ holds whenever two of $x,y,z$ belong to $A_0$.

\section{The main results}

Let $E/F$ be a quadratic extension with nontrivial automorphism $\sigma$.

Firstly let $A$ be the algebra of type V associated with $E/F$.

\begin{thm}
Let $x,y,x',y'\in A$.
Assume $R_{x},R_{y},R_{x'},R_{y'}$ are invertible.
Write $x=(a,b), y=(c,d), x'=(a',b'),y'=(c',d')$ with $a,b,\dots,c',d'\in E$.
Then  $A(x,y)=A(x',y')$ if and only if one of the following holds:

(i) $x'=tx$, $y'=ty$ for some $t\in F^{\times}$.

(ii) $x=my$, $x'=my'$ for some $m\in F^{\times}$.

(iii) $b=d=0$, $b'=d'=0$, $a'=ta$, $c'=tc$ for some $t\in E^{\times}$.

(iv) There exist $l,m, t\in E^{\times}$ such that
\begin{alignat*}{4}
a^{\sigma}&=l b, \quad& b^{\sigma}&=l a, \quad&
a^{\prime\sigma}&=l b',\quad& b^{\prime\sigma}&=l a',\\
c^{\sigma}&=m d, & d^{\sigma}&= mc, \quad&
c^{\prime\sigma}&=m d', & d^{\prime\sigma}&=m c',\\
c&=t a, & d&=t b, \quad&
c' &=t a', & d'&=t b'.
\end{alignat*}
(These equations imply 
$lt^{\sigma}=m t$.)
\end{thm}

(i) and (ii) are the trivial solutions.
(iii) reflects the fact that $A_0=E\oplus 0$ is the weak nucleus of $A$.

Next let $A$ be the algebra of type W associated with  $E/F$.
Regard $A$ as a two-sided vector space over  $E$ by multiplication
$$
k(a,b)=(ka,k^{\sigma}b), \; (a,b)k=(ak,bk).
$$

\begin{thm}
Let $x,y,x',y'\in A$ be nonzero elements.
Then  $A(x,y)=A(x',y')$ if and only if one of the following holds:

(i) 
$x'=l x$, $y'=ly$ for some $l\in E^{\times}$.

(ii) 
$x=ym$, $x'=y'm$ for some $m\in E^{\times}$.

\end{thm}

When $F=\bF_2$, these theorems specialize to those in Introduction.

Theorem 2.1 will be proved in Sections 3--5  through consideration of a split form of the algebra. Theorem 2.2 will be proved directly in Section 6.

\section{The split algebra of type V}

We replace the quadratic Galois extension $E/F$ for the algebra of type V by the split extension $F^2/F$ to obtain a split algebra.
Namely let $E=F^2$, the direct product algebra of two copies of $F$. 
Let $\sigma\colon E\to E$ be the automorphism $(x_1,x_2)\mapsto (x_2,x_1)$.
Let $A=E\oplus E$. Define the multiplication on $A$ by the same rule as (1) of Section 1:
$$
(a,b)(c,d)=(ac+b^{\sigma}d, bc+a^{\sigma}d+b^{\sigma}d^{\sigma})
\quad
\text{for $a,b,c,d\in E$.}
$$
We call $A$ {\it the split algebra of type V}.
 
Write $a=(a_1,a_2)$, $b=(b_1,b_2)$, $c=(c_1,c_2)$, $d=(d_1,d_2)$.
We have
\begin{align*}
&(ac+b^{\sigma}d, bc+a^{\sigma}d+b^{\sigma}d^{\sigma})\\
&=((a_1c_1+b_2d_1, a_2c_2+b_1d_2), (b_1c_1+a_2d_1+b_2d_2, b_2c_2+a_1d_2+b_1d_1))
\end{align*}
and
\begin{align*}
&(a_1c_1+b_2d_1, a_2c_2+b_1d_2, b_1c_1+a_2d_1+b_2d_2, b_2c_2+a_1d_2+b_1d_1)\\
&=(a_1,a_2, b_1,b_2)
\begin{pmatrix}
c_1 & 0 & 0 & d_2 \\
0 & c_2 & d_1 & 0 \\
0 & d_2 & c_1 & d_1 \\
d_1 & 0 & d_2 & c_2 \\
\end{pmatrix}.
\end{align*}
Put
$$
M_{c,d}=
\begin{pmatrix}
c_1 & 0 & 0 & d_2 \\
0 & c_2 & d_1 & 0 \\
0 & d_2 & c_1 & d_1 \\
d_1 & 0 & d_2 & c_2 \\
\end{pmatrix},
$$
so that the right multiplication $R_{(c,d)}$ is represented by the matrix
$M_{c,d}$.
One has
$$
\det M_{c,d}=c_1^2c_2^2-3c_1c_2d_1d_2+d_1^2d_2^2.
$$
So $R_{(c,d)}$ is invertible if and only if $c_1^2c_2^2-3c_1c_2d_1d_2+d_1^2d_2^2\ne 0$.

One sees also the left multiplication $L_{(c,d)}$ has the same determinant.

Now let us back to the algebra $A$ of type V associated with a quadratic Galois extension $E/F$. Let $\sigma$ be the nontrivial automorphism of $E/F$.
Take  a splitting field $K$ of the extension $E/F$, for example $E$ itself.
We have two injections $i_1,i_2\colon E\to K$, so that $i_2=i_1\sigma$.
We have a $K$-algebra isomorphism
$\iota\colon K\otimes E\to  K^2$ taking $
1\otimes a$ to $(i_1(a), i_2(a))$.
Let $\tau$ be the automorphism of $K^2$ taking $(x_1,x_2)$ to $(x_2,x_1)$.
The isomorphism $\iota\colon K\otimes E\to K^2$  transforms $1\otimes \sigma$ into $\tau$. 

Let $\tilde A$ be the split $K$-algebra of type V associated with the  extension $K^2/K$. 
We have a $K$-algebra isomorphism
$K\otimes A\to \tilde A$ taking 
$$
1\otimes (a,b) \mapsto (\iota(1\otimes a), \iota(1\otimes b))
=((i_1(a), i_2(a)), (i_1(b), i_2(b))).
$$
We have
$$
\det R_{(a,b)}=\det M_{(\iota(1\otimes a), \iota(1\otimes b))}=
N(a)^2-3N(a)N(b)+N(b)^2.
$$

\section{Computation in the split algebra of type V}

In this section we solve the equation 
$R_{x}R_{y}^{-1}=R_{x'}R_{y'}^{-1}$ for the split algebra of type V.
The main theorem for the nonsplit algebra of type V will be deduced in the next section from the result here.

Let $A$ be the split algebra of type V associated with the extension $F^2/F$ as defined in Section 3.
Let $M_{c,d}$ be the matrix defined there for any $(c,d)\in A$.
The linear transformation $R_{x}R_{y}^{-1}$ for $x=(a,b)$, $y=(c,d)$ is represented by the matrix $M_{c,d}^{-1}M_{a,b}$. (Note that these matrices operate on row-vectors by right multiplication.)

We say an element $x=(x_1,x_2)\in A$ satisfies condition $(\star)$ if $x_1x_2\ne 0$ or $x_1=x_2=0$.

\begin{thm}
Let $(a,b), (c,d), (a',b'), (c',d')\in A$.
Assume that the matrices $M_{c,d}$, $M_{c',d'}$ are invertible.
Assume that each of $a,b,\dots,c',d'$ satisfies ($\star$).
Then  $M_{c,d}^{-1}M_{a,b}= M_{c',d'}^{-1}M_{a',b'}$ if and only if the one of the following holds:

(i) $(a',b',c',d')=t(a,b,c,d)$ for some $t\in F^{\times}$.

(ii) $(a,b,a',b')=m(c,d,c',d')$ for some $m\in F$.

(iii) $b=b'=0$, $d=d'=0$, $(a_1',c_1')=t_1(a_1,c_1)$, $(a_2',c_2')=t_2(a_2,c_2)$ for some $t_1,t_2\in F^{\times}$.

(iv) There exist $l,l',t_1,t_2\in F^{\times}$ such that
\begin{align*}
(a_2,b_1,c_2,d_1)&=l(b_2,a_1,d_2,c_1),\\
(a_2',b_1',c_2',d_1')&=l'(b_2',a_1',d_2',c_1'),\\
(a_1',b_2',c_1',d_2')&=t_1(a_1,b_2,c_1,d_2),\\
(a_2',b_1',c_2',d_1')&=t_2(a_2,b_1,c_2,d_1).
\end{align*}
(These equations imply $l t_2=l' t_1$.)

\end{thm}

\begin{proof}[Proof of sufficiency]
(i) Let $(a',b',c',d')=t(a,b,c,d)$ with $t\in F^{\times}$. Then
$M_{a',b'}=t M_{a,b},\; M_{c',d'}=t M_{c,d}$,
hence
$
M_{c',d'}^{-1}M_{a',b'} =M_{c,d}^{-1}M_{a,b}$.

(ii) Let $(a,b,a',b')=m(c,d,c',d')$ with $m\in F$.
Then
$M_{a,b}=m M_{c,d} , \; M_{a',b'}=m M_{c',d'}$,
hence
$M_{c,d}^{-1}M_{a,b}=mI=M_{c',d'}^{-1}M_{a',b'}$.

(iii)
Let $b=b'=0$, $d=d'=0$. 
Then
$$
M_{a,b}=\text{diag}(a_1,a_2,a_1,a_2),\;
M_{c,d}=\text{diag}(c_1,c_2,c_1,c_2),
$$
$$
M_{c,d}^{-1}M_{a,b}=\text{diag}(\frac{a_1}{c_1},\frac{a_2}{c_2},\frac{a_1}{c_1},\frac{a_2}{c_2}).
$$
Similarly
$$
M_{c',d'}^{-1}M_{a',b'}=\text{diag}(\frac{a_1'}{c_1'},\frac{a_2'}{c_2'},\frac{a_1'}{c_1'},\frac{a_2'}{c_2'}).
$$
Letting further $(a_1',c_1')=t_1(a_1,c_1)$, $(a_2',c_2')=t_2(a_2,c_2)$
with $t_1,t_2\in F^{\times}$, 
we find the two matrices are equal.

(iv)
Let $(a_2,b_1,c_2,d_1)=l(b_2,a_1,d_2,c_1)$ with $l\in F^{\times}$. Then
$$
M_{a,b}=
\begin{pmatrix}
a_1 & 0 & 0 & b_2 \\
0 & l b_2 & l a_1 & 0 \\
0 & b_2 & a_1 & l a_1\\
l a_1 & 0 & b_2 & l b_2 
\end{pmatrix},\quad
M_{c,d}=
\begin{pmatrix}
c_1 & 0 & 0 & d_2 \\
0 & l d_2 & l c_1 & 0 \\
0 & d_2 & c_1 & l c_1 \\
l c_1 & 0 & d_2 & l d_2 \\
\end{pmatrix},
$$
$$
\det M_{c,d}=-l^2 c_1^2 d_2^2,
$$
$$
M_{c,d}^{-1}M_{a,b}=
\begin{pmatrix}
\dfrac{a_1}{c_1} & 0 & 0 & -\dfrac{a_1d_2-b_2 c_1}{c_1^2}\\
0 & \dfrac{b_2}{d_2} & \dfrac{a_1d_2-b_2c_1}{d_2^2} & 0 \\
0 & 0 & \dfrac{b_2}{d_2} & 0 \\
0 & 0 & 0 & \dfrac{a_1}{c_1}
\end{pmatrix}.
$$
Let $(a_2',b_1',c_2',d_1')=l'(b_2',a_1',d_2',c_1')$ with $l'\in F^{\times}$. Then
$$
M_{c',d'}^{-1}M_{a',b'}=
\begin{pmatrix}
\dfrac{a_1'}{c_1'} & 0 & 0 & -\dfrac{a_1'd_2'-b_2'c_1'}{c_1^{\prime2}}\\
0 & \dfrac{b_2'}{d_2'} & \dfrac{a_1'd_2'-b_2'c_1'}{d_2^{\prime2}} & 0 \\
0 & 0 & \dfrac{b_2'}{d_2'} & 0 \\
0 & 0 & 0 & \dfrac{a_1'}{c_1'}
\end{pmatrix}.
$$
Letting $(a_1',b_2',c_1',d_2')=t_1(a_1,b_2,c_1,d_2)$, we find the two matrices are equal.

Thus any of (i)--(iv) implies the equality
$M_{c,d}^{-1}M_{a,b}= M_{c',d'}^{-1}M_{a',b'}$.

\end{proof}

\begin{proof}[Proof of necessity]
Let $\tilde X$ denote the cofactor matrix of a square matrix $X$, so that
$X^{-1}=(\det X)^{-1}\tilde X$.
Put
$$
H=M_{c,d}^{-1}M_{a,b},\;K=\tilde M_{c,d} M_{a,b},
$$
$$
H=(h_{ij}),\; K=(k_{ij}),
$$
so that
$$
h_{ij}=\frac{1}{\det M_{c,d}}k_{ij}.
$$
Also put
$$
H'= M_{c',d'}^{-1}M_{a',b'},\;
K'=\tilde M_{c',d'} M_{a',b'},
$$
$$
H'=(h'_{ij}),\;
K'=(k_{ij}').
$$
The assumption is that $H=H'$.

The proof proceeds along  division into cases
depending on vanishing of some of the entries $h_{ij}$:

(1) $h_{21}\ne 0$, $h_{12}\ne 0$.

(2) $h_{21}=0$, $h_{12}\ne 0$.

(3) $h_{21}\ne 0$, $h_{12}=0$.

(4) $h_{21}=0$, $h_{12}= 0$.

Each case will be further divided  in the course of the proof.

We also consider cases depending on  vanishing of $c_i$, $d_i$:

(I) $c_1\ne 0, c_2\ne 0, d_1\ne 0, d_2\ne 0$.

(II) $c_1=0, c_2=0, d_1\ne 0, d_2\ne 0$.

(III) $c_1\ne 0, c_2\ne 0, d_1=0, d_2=0$.

By condition ($\star$) these three are only possible.

Correspondingly there are three cases:

(I') $c'_1\ne 0, c'_2\ne 0, d'_1\ne 0, d'_2\ne 0$.

(II') $c'_1=0, c'_2=0, d'_1\ne 0, d'_2\ne 0$.

(III') $c'_1\ne 0, c'_2\ne 0,d'_1=0, d'_2=0$.

Recall
$$
\det M_{c,d}=c_1^2c_2^2-3c_1c_2d_1d_2+d_1^2d_2^2.
$$
The entries of $K$ are given as follows:
\begin{align*}
k_{11}&=a_1 c_1 c_2^2 - b_1 c_1 c_2 d_2 -2 a_1 c_2 d_1 d_2 + b_1 d_1 d_2^2,\\
k_{21}&=-d_1^2 (a_1 d_1 - b_1 c_1),\\ 
k_{31}&=c_2 d_1 (a_1 d_1 - b_1 c_1),\\ 
k_{41}&=-(a_1 d_1- b_1 c_1) (c_1 c_2 - d_1 d_2),
\displaybreak[0]\\
\noalign{\smallskip}
k_{12}&=-d_2^2 (a_2 d_2 - b_2 c_2),\\ 
k_{22}&=a_2 c_1^2 c_2 - b_2 c_1 c_2 d_1 -2 a_2 c_1 d_1 d_2 + b_2 d_1^2 d_2,\\ 
k_{32}&= -(a_2 d_2 -b_2 c_2) (c_1 c_2 - d_1 d_2),\\ 
k_{24}&=c_1 d_2 (a_2 d_2 - b_2 c_2),
\displaybreak[0]\\
\noalign{\smallskip}
k_{13}&=d_2 (-b_2 c_1 c_2 + a_1 c_2 d_2 - b_1 d_2^2 + b_2 d_1 d_2),\\ 
k_{23}&=b_1 c_1^2 c_2 - a_1 c_1 c_2 d_1 -2 b_1 c_1 d_1 d_2 + b_2 c_1 d_1^2 
+ a_1 d_1^2 d_2,\\ 
k_{33}&= a_1 c_1 c_2^2 - b_1 c_1 c_2 d_2 - b_2 c_1 c_2 d_1 - a_1 c_2 d_1 d_2 
+ b_1 d_1 d_2^2,\\ 
k_{43}&=  -c_1 (- b_2 c_1 c_2 + a_1 c_2 d_2 - b_1 d_2^2 + b_2 d_1 d_2),
\displaybreak[0]\\
\noalign{\smallskip}
k_{14}&=b_2 c_1 c_2^2 - a_2 c_1 c_2 d_2 + b_1 c_2 d_2^2 
-2 b_2 c_2 d_1 d_2 + a_2 d_1 d_2^2,\\ 
k_{24}&=  d_1 (-b_1 c_1 c_2 + a_2 c_1 d_1 + b_1 d_1 d_2 - b_2 d_1^2),\\ 
k_{34}&= - c_2 (- b_1 c_1 c_2 + a_2 c_1 d_1 + b_1 d_1 d_2 - b_2 d_1^2),\\ 
k_{44}&=  a_2 c_1^2 c_2 - b_1 c_1 c_2 d_2 - b_2 c_1 c_2 d_1 - a_2 c_1 d_1 d_2 + b_2 d_1^2 d_2.
\tag{1}
\end{align*}

From these one sees readily
\begin{lem}
\begin{align*}
k_{21}\ne 0 &\iff a_1d_1-b_1c_1\ne 0,\; d_1\ne 0.\\
k_{12}\ne 0 &\iff a_2d_2-b_2c_2\ne 0, \;d_2\ne 0.\\
k_{41}\ne 0 &\iff a_1d_1-b_1c_1\ne 0, \;c_1c_2-d_1d_2\ne 0.\\
k_{32}\ne 0 &\iff a_2d_2-b_2c_2\ne 0, \;c_1c_2-d_1d_2\ne 0.\\
\end{align*}
\end{lem}

\begin{lem}
$h_{12}\ne 0,\; h_{21}=0 \implies c_id_ic_i'd_i'\ne 0$.

\end{lem}

\begin{proof}
Suppose $h_{12}\ne0$. Then $k_{12}\ne 0$, $k_{12}'\ne 0$. By the above lemma this  means that
$a_2 d_2-b_2 c_2\ne 0$, $d_2\ne 0$, 
$a_2' d_2'-b_2' c_2'\ne 0$, $d_2'\ne 0$.
By condition  ($\star$) we then have
$d_1\ne 0$, $d_1'\ne 0$.
Suppose further $h_{21}=0$. Then $k_{21}=0$, $k_{21}'=0$. By the lemma this implies that
$a_1d_1-b_1c_1=0$, $a_1'd_1'-b_1'c_1'=0$.
Suppose further $c_1=0$, $c_2=0$. Then
$a_1d_1=0$, hence 
$a_1=0$. By condition ($\star$) we have $a_2=0$.
Then $a_2d_2-b_2c_2=0$, a contradiction.
We must have $c_1\ne 0$, $c_2\ne 0$ by ($\star$).

Similarly $c_1'\ne 0$, $c_2'\ne 0$.

\end{proof}

\begin{lem}

\begin{align*}
&h_{21}=0,\; c_1'=c_2'=0, \;d_1'd_2'\ne 0 \implies h_{23}=0.\\
&h_{21}=0, \;d_1d_2\ne 0, \;c_1'c_2'\ne 0, \;d_1'=d_2'=0 \implies h_{23}=0.
\end{align*}

\end{lem}

\begin{proof}
Suppose
$c_1'=0,\; c_2'=0,\;d_1'\ne 0,\; d_2'\ne 0$.
By (1) we have
$k_{21}'=-a_1'd_1^{\prime 3}$,
$k_{23}'=a_1'd_1^{\prime2}d_2'$.
Suppose further that $h_{21}=0$. Then $k_{21}'=0$, so $a_1'=0$, hence $k_{23}'=0$. Thus $h_{23}=0$. 

Suppose next that 
$c_1'\ne 0, \; c_2'\ne 0, \; d_1'=0, \; d_2'=0$.
Then, by (1),
$k'_{41}=k'_{23}=b_1'c_1^{\prime 2}c_2'$.
Suppose further that $h_{21}=0$ and
$d_1\ne 0$, $d_2\ne 0$.
Then, by Lemma 4.2,
$a_1d_1-b_1c_1=0$. Hence $h_{41}=0$, so $k'_{41}=0$.
Then $k_{23}'=0$,  $h_{23}=0$. 
\end{proof}

Now we examine the four cases in order.

Case (1) $h_{21}\ne 0$, $h_{12}\ne 0$.
Then
\begin{align*}
&a_1d_1-b_1c_1\ne 0,\; a_2d_2-b_2c_2\ne 0, \; d_1\ne 0, \; d_2\ne 0,\\
&a_1'd_1'-b_1'c_1'\ne 0, \; a_2'd_2'-b_2'c_2'\ne 0,\; d_1'\ne 0,\; d_2'\ne 0.
\end{align*}
By (1) 
\begin{align*}
\frac{k_{42}}{k_{12}}&=-\frac{c_1}{d_2}, \displaybreak[0]\\
\frac{k_{31}}{k_{21}}&=-\frac{c_2}{d_1}, \displaybreak[0]\\
\frac{k_{41}}{k_{21}}&=\frac{c_1c_2-d_1d_2}{d_1^2}, \displaybreak[0]\\
\frac{k_{32}}{k_{12}}&=\frac{c_1c_2-d_1d_2}{d_2^2}.
\end{align*}
By the assumption $H=H'$ it follows that
\begin{align*}
\frac{c_1}{d_2}&=\frac{c_1'}{d_2'},\tag2 \displaybreak[0]\\
\frac{c_2}{d_1}&=\frac{c_2'}{d_1'},\tag3 \displaybreak[0]\\
\frac{c_1c_2-d_1d_2}{d_1^2}&=\frac{c_1'c_2'-d_1'd_2'}{d_1^{\prime2}}, \tag4 \displaybreak[0]\\
\frac{c_1c_2-d_1d_2}{d_2^2}&=\frac{c_1'c_2'-d_1'd_2'}{d_2^{\prime2}}. \tag5\\
\end{align*}

We distinguish two cases:

(1.1) $h_{41}\ne 0$.

(1.2) $h_{41}=0$.

Case (1.1) $h_{41}\ne 0$.
Then
$c_1c_2-d_1d_2\ne 0, \; c_1'c_2'-d_1'd_2'\ne 0$, and
$$
\frac{k_{41}}{k_{21}}/\frac{k_{32}}{k_{12}}=\frac{d_2^2}{d_1^2}.
$$
Hence
$$
\frac{d_2^2}{d_1^2}
=\frac{d_2^{\prime2}}{d_1^{\prime2}}.
$$
So we have either
\begin{align*}
\frac{d_2}{d_1}&= \frac{d_2'}{d_1'}\tag6\\
\text{or}\quad
\frac{d_2}{d_1}&= -\frac{d_2'}{d_1'}. \tag7
\end{align*}

Case where (6) holds.
 (2), (3), (6) show that
$(c_1,c_2,d_1,d_2)$ is proportional to  $(c_1',c_2',d_1',d_2')$.
We can write
$(c',d')=t (c,d)$
with $t\in F^\times$.
Then $M_{c',d'}=t M_{c,d}$. As $M_{c,d}^{-1}M_{a,b}=M_{c',d'}^{-1}M_{a',b'}$, we have
$M_{a',b'}=t M_{a,b}$.
Then 
$(a',b')=t (a,b)$.
Thus (i) of the theorem holds.

Case where (7) holds  and $\text{char}(F)\ne 2$.
We can write $d_1'=t d_1$, $d_2'=-t d_2$ with $t\in F^{\times}$. Then, by (2) and (3), 
$c_1'=-tc_1$, $c_2'=t c_2$.
Then
$$
\text{RHS of (4)}=\frac{t(-t)(c_1c_2-d_1d_2)}{t^2d_1^2}
=-\frac{c_1c_2-d_1d_2}{d_1^2}
=-\text{(LHS of (4))}.
$$
As $\text{char}(F)\ne 2$,  the LHS of (4) must be zero, namely
$c_1c_2-d_1d_2=0$. This contradicts the assumption of the present case.

This settles Case (1.1).

\medskip
Case (1.2) $h_{41}=0$.
Then 
$c_1c_2-d_1d_2=0$, 
so
$c_1c_2=d_1d_2\ne 0$.
We can write
$d_2=l c_1$, $d_1=l^{-1}c_2$ with $l\in F^{\times}$.
Similarly
$c_1'c_2'=d_1'd_2'\ne 0$.
By (2) and (3) we have
$d_2'=l c_1'$, $d_1'=l^{-1}c_2'$.

By (1) we have  
\begin{align*}
k_{11}&=-a_1c_1c_2^2, \displaybreak[0]\\
k_{21}&=-l^{-3}c_2^2(a_1c_2-l b_1c_1), \displaybreak[0]\\
k_{22}&=-a_2c_1^2c_2, \displaybreak[0]\\
k_{33}-k_{11}&=c_1c_2^2(a_1-l^{-1}b_2), \displaybreak[0]\\
k_{44}-k_{22}&=c_1^2c_2(a_2-l b_1).\\
\end{align*}
And
$$
\det M_{c,d}=-c_1^2c_2^2.
$$
Hence
\begin{align*}
h_{11}&=\frac{a_1}{c_1}, \displaybreak[0]\\
h_{22}&=\frac{a_2}{c_2}, \displaybreak[0]\\
h_{33}-h_{11}&=-\frac{a_1-l^{-1}b_2}{c_1}, \displaybreak[0]\\
h_{44}-h_{22}&=-\frac{a_2-lb_1}{c_2}, \displaybreak[0]\\
h_{21}&=l^{-3}(\frac{a_1}{c_1}-l \frac{b_1}{c_2})\frac{c_2}{c_1}.
\end{align*}
By the assumption $H=H'$ it follows that
\begin{align*}
\frac{a_1}{c_1}&=\frac{a_1'}{c_1'},\tag8 \displaybreak[0]\\
\frac{a_2}{c_2}&=\frac{a_2'}{c_2'},\tag9 \displaybreak[0]\\
\frac{a_1-l^{-1}b_2}{c_1}&=\frac{a_1'-l^{-1}b_2'}{c_1'},\tag{10} \displaybreak[0]\\
\frac{a_2-lb_1}{c_2}&=\frac{a_2'-lb_1'}{c_2'},\tag{11} \displaybreak[0]\\
l^{-3}(\frac{a_1}{c_1}-l \frac{b_1}{c_2})\frac{c_2}{c_1}&=
l^{-3}(\frac{a_1'}{c_1'}-l \frac{b_1'}{c_2'})\frac{c_2'}{c_1'}.\tag{12}
\end{align*}
By (8), (9), (10), (11)
\begin{align}
\frac{b_2}{c_1}&=\frac{b_2'}{c_1'}, \tag{13}\\
\frac{b_1}{c_2}&=\frac{b_1'}{c_2'}. \tag{14}
\end{align}
By (8), (14),  (12) and the assumption $h_{21}\ne 0$
\begin{align*}
\frac{c_2}{c_1}=\frac{c_2'}{c_1'}. \tag{15}
\end{align*}
By (2), (3),  (8), (9), (13), (14), (15)  we see that
$(a,b,c,d)$ is proportional to $(a',b',c',d')$.
Thus (i) of the theorem holds.

This settles Case (1.2).

Consequently Case (1) results in (i) of the theorem.

\medskip
Case (2) $h_{21}=0$, $h_{12}\ne 0$.
By Lemmas 4.2 and 4.3 we have
\begin{align*}
&a_1d_1-b_1c_1=0,\; a_2 d_2-b_2 c_2\ne 0,\;c_1,c_2, d_1,d_2\ne 0,\\
&a_1'd_1'-b_1'c_1'=0, \; a_2' d_2'-b_2' c_2'\ne 0,\; c_1',c_2',d_1',d_2'\ne 0.
\end{align*}
We can write
$(b_1,d_1)=l (a_1,c_1)$, 
$(b_1',d_1')=l' (a_1',c_1')$
with $l,l'\in F^{\times}$.
Then
$$
\det M_{c,d}=c_1^2(c_2^2-3l c_2d_2+l^2 d_2^2).
$$
The entries of $K$ are given as follows.
\begin{align*}
k_{11}&=a_1c_1(c_2^2-3lc_2d_2+l^2 d_2^2),\\
k_{21}&=0,\\
k_{31}&=0,\\
k_{41}&=0,
\displaybreak[0]\\
\noalign{\smallskip}
k_{12}&=-d_2^2(a_2d_2-b_2c_2),\\
k_{22}&=c_1^2(a_2c_2-lb_2c_2-2l a_2d_2+l^2b_2d_2),\\
k_{32}&=-c_1(a_2d_2-b_2c_2)(c_2-ld_2),\\
k_{42}&=c_1d_2(a_2d_2-b_2c_2),
\displaybreak[0]\\
\noalign{\smallskip}
k_{13}&=d_2(a_1d_2-b_2c_1)(c_2-ld_2),\\
k_{23}&=-l^2c_1^2(a_1d_2-b_2c_1),\\
k_{33}&=c_1(a_1c_2^2-l b_2c_1c_2-2l a_1c_2d_2+l^2a_1d_2^2),\\
k_{43}&=-c_1(a_1d_2-b_2c_1)(c_2-ld_2),
\displaybreak[0]\\
\noalign{\smallskip}
k_{14}&=-c_1c_2(a_2d_2-b_2c_2)-2l b_2c_1c_2d_2+ld_2^2(a_1c_2+a_2c_1),\\
k_{24}&=-l^2c_1^2(a_1c_2-a_2c_1-l(a_1d_2-b_2c_1)),\\
k_{34}&=lc_1c_2(a_1c_2-a_2c_1-l(a_1d_2-b_2c_1)),\\
k_{44}&=c_1(a_2c_1c_2-ld_2(a_1c_2+a_2c_1)-lb_2c_1(c_2-ld_2)).
\tag{16}
\end{align*}
In particular
\begin{align*}
\frac{k_{42}}{k_{12}}&=-\frac{c_1}{d_2},\\
\frac{k_{32}}{k_{42}}&=-\frac{c_2}{d_2}+l.
\end{align*}
Hence
\begin{align*}
\frac{c_1}{d_2}&=\frac{c_1'}{d_2'},\tag{17}\\
-\frac{c_2}{d_2}+l&=-\frac{c_2'}{d_2'}+l'. \tag{18}
\end{align*}
By (17) we can write
$d_2=m c_1$,  $d_2'=m c_1'$ with $m\in F^{\times}$.
Then (18) becomes
\begin{align*}
-\frac{1}{m}\frac{c_2}{c_1}+l=-\frac{1}{m}\frac{c_2'}{c_1'}+l'. \tag{19}
\end{align*}
We distinguish two cases:

(2.1) $h_{23}\ne 0$.

(2.2) $h_{23}=0$.

Case (2.1) $h_{23}\ne 0$.  Then
$a_1d_2-b_2c_1\ne 0$ and
$$
\frac{k_{13}}{k_{23}}=-\frac{d_2(c_2-ld_2)}{l^2c_1^2}
=-\frac{1}{l^2}\frac{d_2^2}{c_1^2}(\frac{c_2}{d_2}-l)
=\frac{m^2}{l^2}(-\frac{c_2}{d_2}+l).
$$
We further distinguish two cases:

(2.1.1) $h_{32}\ne 0$.

(2.1.2) $h_{32}=0$.

Case (2.1.1) $h_{32}\ne 0$.
Then $0\ne c_2-ld_2=c_2-lm c_1$ and
$$
\frac{k_{13}}{k_{23}}/\frac{k_{32}}{k_{42}}=\frac{m^2}{l^2}.
$$
Hence
$$
\frac{m^2}{l^2}=\frac{m^2}{l^{\prime2}}.
$$
We have either
\begin{align*}
l&= l'\tag{20}\\
\text{or}\quad l&=-l'. \tag{21}
\end{align*}

Case where (20) holds.
Then, by (19)
$$
\frac{c_2}{c_1}=\frac{c_2'}{c_1'}. 
$$
We can write $c_2=n c_1$, $c_2'=n c_1'$ with $n\in F^{\times}$.
The substitution 
$c_2=n c_1$, $d_2=m c_1$ in (16) gives
\begin{align*}
k_{11}&=(l^2m^2-3lmn+n^2)a_1c_1^3,\\
k_{22}&=-c_1^3((2lm-n)a_2-(lm-n)lb_2),\\
k_{42}&=mc_1^3(ma_2-nb_2),\\
k_{23}&=-l^2c_1^3(ma_1-b_2).
\end{align*}
And
$$
\det M_{c,d}=(l^2m^2-3lmn+n^2)c_1^4,
$$
$$
mk_{22}+\frac{2lm-n}{m}k_{42}=(l^2m^2-3lmn+n^2)b_2c_1^3.
$$
It follows that
\begin{align*}
h_{11}&=\frac{a_1}{c_1},\\
h_{42}&=\frac{m}{l^2m^2-3lmn+n^2} \frac{ma_2-nb_2}{c_1},\\
mh_{22}+\frac{2lm-n}{m}h_{42}&=\frac{b_2}{c_1}.
\end{align*}
Hence
\begin{align*}
\frac{a_1}{c_1}&=\frac{a_1'}{c_1'},\tag{22}\\
\frac{ma_2-nb_2}{c_1}&=\frac{ma_2'-nb_2'}{c_1'},\tag{23}\\
\frac{b_2}{c_1}&=\frac{b_2'}{c_1'}.\tag{24}
\end{align*}
By (23), (24)
\begin{align*}
\frac{a_2}{c_1}=\frac{a_2'}{c_1'}.\tag{25}
\end{align*}
(22), (24), (25) show that
$(a_1,a_2,b_2,c_1)$ is proportional to  $(a_1',a_2',b_2',c_1')$.
Recalling the equalities
$(b_1,d_1)=l(a_1,c_1)$, $(b_1',d_1')=l(a_1',c_1')$,
$d_2=mc_1$, $d_2'=mc_1'$,
$c_2=nc_1$, $c_2'=nc_1'$,
we conclude that $(a, b,c,d)$ is proportional to  $(a',b',c',d')$.
Thus (i) of the theorem holds.

Case where (21) holds and $\text{char}(F)\ne 2$.
We distinguish cases according as $h_{24}=0$ or not.

Case where  $h_{24}\ne 0$. Then, by (16),
$$
\frac{k_{24}}{k_{34}}=-l \frac{c_1}{c_2}.
$$
Hence
$$
 -l \frac{c_1}{c_2}=-l' \frac{c_1'}{c_2'}.
$$
By (21) we have
$$
\frac{c_1}{c_2}=-\frac{c_1'}{c_2'} 
$$
and so
$$
-\frac{1}{m}\frac{c_2}{c_1}+l=-(-\frac{1}{m}\frac{c_2'}{c_1'}+l').
$$
Compared with (19), this gives  
$$
-\frac{1}{m}\frac{c_2}{c_1}+l=0,
$$
which contradicts the assumption $c_2-l m c_1\ne 0$ of Case (2.1.1).

Thus this case is excluded.

\medskip
Case where $h_{24}= 0$.
Then, by (16),  $a_1c_2-a_2c_1=l(a_1d_2-b_2c_1)$. 
Using this, we compute
\begin{align*}
k_{33}-k_{44}&=c_1((a_1c_2-a_2c_1)(c_2-ld_2)+l^2(a_1d_2-b_2c_1)d_2)\\
&=c_1(l(a_1d_2-b_2c_1)(c_2-ld_2)+l^2(a_1d_2-b_2c_1)d_2)\\
&=lc_1c_2(a_1d_2-b_2c_1).
\end{align*}
As we are in Case (2.1),
$k_{23}=-l^2c_1^2(a_1d_2-b_2c_1)\ne 0$,
so
$$
\frac{k_{33}-k_{44}}{k_{23}}=-\frac{1}{l}\frac{c_2}{c_1}.
$$
Hence
$$
-\frac{1}{l}\frac{c_2}{c_1}=-\frac{1}{l'}\frac{c_2'}{c_1'}.
$$
By (21),
\begin{align*}
\frac{c_2}{c_1}=-\frac{c_2'}{c_1'}
\end{align*}
and so
$$
-\frac{1}{m}\frac{c_2}{c_1}+l=-(-\frac{1}{m}\frac{c_2'}{c_1'}+l').
$$
This and (19)  again yield 
$c_2-l m c_1=0$, 
contradicting the assumption of Case (2.1.1).

This settles Case (2.1.1).

\medskip
Case (2.1.2) $h_{32}=0$. Then 
$c_2=l d_2=l m c_1$, $c_2'=l' d_2'=l' m c_1'$ and
$0\ne a_2d_2-b_2c_2=(a_2-l b_2)d_2$, 
so
$a_2-l b_2\ne 0$.
Similarly
$a_2'-l' b_2'\ne 0$.
Substituting $d_2=m c_1$, $c_2=l m c_1$ in (16),
we have 
\begin{align*}
k_{11}&=-l^2 m^2 a_1c_1^3, \\
k_{12}&=-m^3 c_1^3(a_2-l b_2),\\
k_{22}&=-l m a_2c_1^3,\\
k_{33}&=-l^2 m b_2c_1^3,\\
k_{24}&=l^2c_1^3(a_2-l b_2).
\end{align*}
And
$$
\det M_{c,d}=-l^2m^2c_1^4,
$$
$$
\frac{k_{12}}{k_{24}}=-\frac{m^3}{l^2}.
$$
Hence
$$
-\frac{m^3}{l^2}=-\frac{m^3}{l^{\prime2}}
$$
so
$$
l^2=l^{\prime 2}.
$$
We have either
\begin{align*}
l&= l' \tag{26}\\
\text{or}\quad l&=-l'. \tag{27}
\end{align*}
Also we have
\begin{align*}
h_{11}&=\frac{a_1}{c_1},\\
h_{22}&=\frac{1}{lm}\frac{a_2}{c_1},\\
h_{33}&=\frac{1}{m}\frac{b_2}{c_1}.
\end{align*}
Hence
\begin{align*}
\frac{a_1}{c_1}&=\frac{a_1'}{c_1'},\tag{28}\\
\frac{1}{l}\frac{a_2}{c_1}&=\frac{1}{l'}\frac{a_2'}{c_1'},\tag{29}\\
\frac{b_2}{c_1}&=\frac{b_2'}{c_1'}.\tag{30}
\end{align*}

Case where (26) holds. 
Then, by (29)
\begin{align*}
\frac{a_2}{c_1}=\frac{a_2'}{c_1'}. \tag{31}
\end{align*}
(28), (30), (31) show that
$(a_1,a_2,b_2,c_1)$ is proportional to $(a_1',a_2',b_2',c_1')$.
Recalling that
$(b_1,d_1)=l(a_1,c_1)$, $(b_1',d_1')=l(a_1',c_1')$, $c_2=l d_2$, $d_2=m c_1$,
$c_2'=l d_2'$, $d_2'=m c_1'$, 
we conclude that
$(a,b,c,d)$ is proportional to $(a',b',c',d')$.

Case where (27) holds and $\text{char}(F)\ne 2$.
We have
$$
h_{12}=\frac{m}{l^2}(\frac{a_2}{c_1}-l\frac{b_2}{c_1}).
$$
Hence
$$
\frac{m}{l^2}(\frac{a_2}{c_1}-l\frac{b_2}{c_1})=
\frac{m}{l^{\prime2}}(\frac{a_2'}{c_1'}-l'\frac{b_2'}{c_1'}).
$$
This and  (27) give
\begin{align*}
\frac{a_2}{c_1}-l\frac{b_2}{c_1}=\frac{a_2'}{c_1'}-l'\frac{b_2'}{c_1'}.\tag{32}
\end{align*}
By (29)
$$
\frac{a_2}{c_1}=-\frac{a_2'}{c_1'}.
$$
This and (27), (30) give
$$
\frac{a_2}{c_1}-l\frac{b_2}{c_1}=-(\frac{a_2'}{c_1'}-l'\frac{b_2'}{c_1'}).
$$
This and (32) imply
$$
\frac{a_2}{c_1}-l\frac{b_2}{c_1}=0,
$$
which contradicts $a_2-lb_2\ne 0$.

This settles Case (2.1.2).

We finish Case (2.1).

\medskip
Case (2.2) $h_{23}=0$. Then $a_1d_2-b_2c_1=0$, $a_1'd_2'-b_2'c_1'=0$.
Recall   
$c_1c_2d_1c_2\ne 0$, $c_1'c_2'd_1'd_2'\ne 0$, and
$d_2=m c_1$, $d_2'=m c_1'$.
So we have
$(b_2,d_2)=m(a_1,c_1)$, $(b_2',d_2')=m(a_1',c_1')$.
Then
$0\ne a_2d_2-b_2c_2=m(a_2c_1-a_1c_2)$, 
so
$a_2c_1-a_1c_2\ne 0$.

Substituting the above expressions for $b_2,d_2$ into (16), we have 
\begin{align*}
k_{11}&=a_1c_1(c_2^2-3lm c_1c_2+l^2m^2 c_1^2),\\
k_{12}&=m^3c_1^2(a_1c_2-a_2c_1),\\
k_{22}&=-c_1^2(-a_2c_2+2lm a_2c_1+lm a_1c_2-l^2m^2 a_1c_1),\\
k_{42}&=-m^2c_1^2(a_1c_2-a_2c_1),\\
k_{24}&=-l^2c_1^2(a_1c_2-a_2c_1),\\
k_{34}&=lc_1c_2(a_1c_2-a_2c_1),\\
k_{44}&=-c_1^2(-a_2c_2+lm a_2c_1+2lm a_1c_2-l^2m^2 a_1c_1).
\end{align*}
And we have
$$
\det M_{c,d}=c_1^2(c_2^2-3lm c_1c_2+l^2m^2 c_1^2).
$$
So
$$
h_{11}=\frac{a_1}{c_1}.
$$
Hence
\begin{align*}
\frac{a_1}{c_1}=\frac{a_1'}{c_1'}.\tag{33}
\end{align*}
Also
$$
k_{22}-k_{44}=lm c_1^2(a_1c_2-a_2c_1),
$$
$$
\frac{k_{22}-k_{44}}{k_{24}}=-\frac{m}{l}.
$$
Hence
$$
\frac{m}{l}=\frac{m}{l'},
$$
so
\begin{align*}
l=l'.\tag{34}
\end{align*}
We have
$$
\frac{k_{34}}{k_{42}}=-\frac{l}{m^2}\frac{c_2}{c_1}.
$$
Hence
\begin{align*}
\frac{l}{m^2}\frac{c_2}{c_1}=\frac{l'}{m^2}\frac{c_2'}{c_1'},
\end{align*}
so
\begin{align*}
\frac{c_2}{c_1}=\frac{c_2'}{c_1'}.\tag{35}
\end{align*}
And we have
\begin{align*}
\frac{k_{11}}{k_{12}}&=\frac{a_1c_1(c_2^2-3lm c_1c_2+l^2m^2 c_1^2)}
{m^3c_1^2(a_1c_2-a_2c_1)}\\
&=\frac{1}{m^3}\frac{a_1}{c_1}\frac{c_2}{c_1}\frac{1-3lm\frac{c_1}{c_2}+l^2m^2\frac{c_1^2}{c_2^2}}{\frac{a_1}{c_1}-\frac{a_2}{c_2}}.
\end{align*}
Hence
\begin{align*}
\frac{1}{m^3}\frac{a_1}{c_1}\frac{c_2}{c_1}\frac{1-3lm\frac{c_1}{c_2}+l^2m^2\frac{c_1^2}{c_2^2}}{\frac{a_1}{c_1}-\frac{a_2}{c_2}}=
\frac{1}{m^3}\frac{a_1'}{c_1'}\frac{c_2'}{c_1'}\frac{1-3l'm\frac{c_1'}{c_2'}+l^{\prime2}m^2\frac{c_1^{\prime2}}{c_2^{\prime2}}}{\frac{a_1'}{c_1'}-\frac{a_2'}{c_2'}}.
\tag{36}
\end{align*}
(33), (34), (35), (36) give
\begin{align*}
\frac{a_2}{c_2}=\frac{a_2'}{c_2'}.
\end{align*}
Consequently $(a_1,a_2,c_1,c_2)$ is proportional to $(a_1',a_2',c_1',c_2')$.
As we know
\begin{alignat*}{2}
(b_2,d_2)&=m(a_1,c_1), \quad&(b_2',d_2')&=m(a_1',c_1'),\\
(b_1,d_1)&=l(a_1,c_1), & (b_1',d_1')&=l(a_1',c_1'),
\end{alignat*}
it follows that
$(a,b,c,d)$ is proportional to $(a',b',c',d')$.

Thus (i) of the theorem holds.

This settles Case (2.2).

Thus Case (2) also results in (i) of the theorem.

\medskip
Case (3) $h_{21}\ne 0$, $h_{12}=0$. This case is reduced to Case (2).
The automorphism $\sigma\colon(x_1,x_2)\mapsto (x_2,x_1)$ of $F^2$ induces an automorphism $(a,b)\mapsto (a^{\sigma}, b^{\sigma})$ of $A$, which we denote by 
$\tilde\sigma$.
This is represented by the matrix
$$
\begin{pmatrix}
0 & 1 & 0 & 0\\
1 & 0 & 0 & 0\\
0 & 0 & 0 & 1\\
0 & 0 & 1 & 0
\end{pmatrix},
$$
which we denote by $T$.
Since $R_{x^{\tilde\sigma}}\tilde\sigma=\tilde\sigma R_{x}$, we have
$T M_{a^{\sigma}, b^{\sigma}}=M_{a,b} T$.
The conjugation by $T$ transforms the equation $M_{c,d}^{-1}M_{a,b}=M_{c',d'}^{-1}M_{a',b'}$ into the equation
$M_{c^{\sigma},d^{\sigma}}^{-1}M_{a^{\sigma},b^{\sigma}}=M_{c^{\prime\sigma},d^{\prime\sigma}}^{-1}M_{a^{\prime\sigma},b^{\prime\sigma}}$, 
and  interchanges the (1,2)-entry and (2,1)-entry of a matrix.

Therefore Case (3) is translated to Case (2) for $a^{\sigma}, b^{\sigma}, c^{\sigma}, d^{\sigma}$.
So (i) of the theorem holds for $a^{\sigma}, b^{\sigma}, c^{\sigma}, d^{\sigma}$, hence for
$a,b,c,d$.

\medskip
Case (4) $h_{21}=0$, $h_{12}=0$.
As observed earlier, one of Cases (I), (II), (III) occurs
and  one of Cases (I'), (II'), (III') occurs.
 By combination of these cases 
there arise six cases:

(4-I) $c_1c_2d_1d_2\ne 0$.

(4-I') $c_1'c_2'd_1'd_2'\ne 0$.

(4-II-II') $d_1d_2\ne 0$, $c_1=c_2=0$,  $d_1'd_2'\ne 0$, $c_1'=c_2'=0$.

(4-III-III') $c_1c_2\ne 0$, $d_1=d_2=0$, $c_1'c_2'\ne 0$, $d_1'=d_2'=0$.

(4-II-III') $d_1d_2\ne 0$, $c_1=c_2=0$, $c_1'c_2'\ne 0$, $d_1'=d_2'=0$.

(4-III-II') $c_1c_2\ne 0$, $d_1=d_2=0$, $d_1'd_2'\ne 0$, $c_1'=c_2'=0$.

Case (4-I) $c_1c_2d_1d_2\ne 0$.
Then, by Lemma 4.2 we have
$a_1d_1-b_1c_1=0$, $a_2d_2-b_2c_2=0$.
We can write
$(b_1,d_1)=l_1(a_1,c_1)$, $(b_2,d_2)=l_2(a_2,c_2)$
with $l_1,l_2\in F^{\times}$.
By (1) we have 
\begin{align*}
k_{11}&=(1-3l_1l_2+l_1^2l_2^2)a_1c_1c_2^2,\\
k_{22}&=(1-3l_1l_2+l_1^2l_2^2)a_2c_1^2c_2,\\
k_{13}&=l_2^2(1-l_1l_2)c_2^2(a_1c_2-a_2c_1),\\
k_{23}&=-l_1^2l_2c_1^2(a_1c_2-a_2c_1),\\
k_{43}&=-l_2(1-l_1l_2)c_1c_2(a_1c_2-a_2c_1),\\
k_{14}&=l_1l_2^2c_2^2(a_1c_2-a_2c_1),\\
k_{24}&=-l_1^2(1-l_1l_2)c_1^2(a_1c_2-a_2c_1),\\
k_{34}&=l_1(1-l_1l_2)c_1c_2(a_1c_2-a_2c_1).
\tag{36}
\end{align*}
And
$$
\det M_{c,d}=(1-3l_1l_2+l_1^2l_2^2)c_1^2c_2^2.
$$

We distinguish two cases:

(4-I.1) $h_{23}\ne 0$.

(4-I.2) $h_{23}=0$.

Case (4-I.1) $h_{23}\ne 0$. Then
$a_1c_2-a_2 c_1\ne 0$.
By Lemma 4.4  Case (I') must occur:
$c_1'c_2'd_1'd_2'\ne 0$.
Then, by Lemma 4.2,
$a_1'd_1'-b_1'c_1'=0$, $a_2'd_2'-b_2'c_2'=0$.
We can write
$(b_1',d_1')=l_1'(a_1',c_1')$, 
$(b_2',d_2')=l_2'(a_2',c_2')$
with $l_1',l_2'\in F^{\times}$.
By (36)
\begin{align*}
\frac{k_{24}}{k_{23}}&=\frac{1}{l_2}-l_1,\\
\frac{k_{13}}{k_{14}}&=\frac{1}{l_1}-l_2.
\end{align*}
Hence
\begin{align*}
\frac{1}{l_2}-l_1&=\frac{1}{l_2'}-l_1', \tag{37}\\
\frac{1}{l_1}-l_2&=\frac{1}{l_1'}-l_2'. \tag{38}
\end{align*}

We distinguish two cases: 

(4-I.1.1) $h_{24}\ne 0$.

(4-I.1.2) $h_{24}=0$.

Case (4-I.1.1) $h_{24}\ne 0$. Then $1-l_1l_2\ne 0$,
$1-l_1'l_2'\ne 0$, and
$$
\frac{k_{43}}{k_{34}}=-\frac{l_2}{l_1}.
$$
Hence
\begin{align*}
\frac{l_2}{l_1}=\frac{l_2'}{l_1'}.\tag{39}
\end{align*}
By (36)
\begin{align*}
h_{11}&=\frac{a_1}{c_1},\\
h_{22}&=\frac{a_2}{c_2}.
\end{align*}
Hence
\begin{align*}
\frac{a_1}{c_1}&=\frac{a_1'}{c_1'},\tag{40}\\
\frac{a_2}{c_2}&=\frac{a_2'}{c_2'}.\tag{41}
\end{align*}
Also by (36)
\begin{align*}
\frac{k_{13}}{k_{24}}&=-\frac{l_2^2}{l_1^2}\frac{c_2^2}{c_1^2},\\
\frac{k_{13}}{k_{43}}&=-l_2 \frac{c_2}{c_1},\\
\frac{k_{24}}{k_{34}}&=-l_1 \frac{c_1}{c_2}.
\end{align*}
Hence
\begin{align*}
\frac{l_2^2}{l_1^2}\frac{c_2^2}{c_1^2}&=\frac{l_2^{\prime2}}{l_1^{\prime2}}\frac{c_2^{\prime2}}{c_1^{\prime2}},\tag{42}\\
l_2 \frac{c_2}{c_1}&=l_2' \frac{c_2'}{c_1'},\tag{43}\\
l_1 \frac{c_1}{c_2}&=l_1' \frac{c_1'}{c_2'}.\tag{44}
\end{align*}
By (42) and (39) we have
$$
\frac{c_2^2}{c_1^2}=\frac{c_2^{\prime2}}{c_1^{\prime2}},
$$
so either
\begin{align*}
\frac{c_2}{c_1}&=\frac{c_2'}{c_1'}\tag{45}\\
\text{or}\quad\frac{c_2}{c_1}&=-\frac{c_2'}{c_1'}.\tag{46}
\end{align*}

Case where  (45) holds.
Then, by (43), (44) we have
$l_1=l_1'$, $l_2=l_2'$. 
(40), (41), (45) show that
$(a_1,a_2,c_1,c_2)$ is proportional to $(a_1',a_2',c_1',c_2')$, that is, 
$(a,c)$ is proportional to $(a',c')$.
Write
$(a',c')=t (a,c)$ with $t\in F^{\times}$.
Since
$(b_i,d_i)=l_i(a_i,c_i)$, $(b_i',d_i')=l_i'(a_i',c_i')$,
we have
$(b',d')=t (b,d)$.
Thus
$(a',b',c',d')=t (a,b,c,d)$.
(i) of the theorem holds.

Case where (46) holds and $\text{char}(F)\ne 2$.
By (43), (44) we have
$l_1=-l_1'$, $l_2=-l_2'$.
Then (37) becomes
$$
\frac{1}{l_2}-l_1=-(\frac{1}{l_2}-l_1),
$$
hence
$$
\frac{1}{l_2}-l_1=0.
$$
This contradicts the assumption of Case (4-I.1.1).

This settles (4-I.1.1).

Case (4-I.1.2)
$h_{24}=0$. Then
$1-l_1l_2=0$, $1-l_1'l_2'=0$, and
$$
\det M_{c,d}=-c_1^2c_2^2.
$$
By (36) we have 
\begin{align*}
k_{11}&=-a_1c_1c_2^2,\\
k_{22}&=-a_2c_1^2c_2,\\
k_{23}&=-l_1c_1^2(a_1c_2-a_2c_1).
\end{align*}
So
\begin{align*}
h_{11}&=\frac{a_1}{c_1},\\
h_{22}&=\frac{a_2}{c_2},\\
h_{23}&=-l_1\frac{c_1}{c_2}(\frac{a_1}{c_1}-\frac{a_2}{c_2}).
\end{align*}
Hence
\begin{align*}
\frac{a_1}{c_1}&=\frac{a_1'}{c_1'},\\
\frac{a_2}{c_2}&=\frac{a_2'}{c_2'},\\
l_1\frac{c_1}{c_2}(\frac{a_1}{c_1}-\frac{a_2}{c_2})&=
l_1'\frac{c_1'}{c_2'}(\frac{a_1'}{c_1'}-\frac{a_2'}{c_2'}).
\end{align*}
Recall that we are in Case (4.1): $h_{23}\ne 0$, i.e., $\frac{a_1}{c_1}-\frac{a_2}{c_2}\ne 0$.
It follows that
\begin{align*}
l_1\frac{c_1}{c_2}=l_1'\frac{c_1'}{c_2'}.
\end{align*}
Put
$\frac{c_1'}{c_1}=t_1$, $\frac{c_2'}{c_2}=t_2$. 
Then
$$
(a_1',c_1')=t_1(a_1,c_1), \;(a_2',c_2')=t_2(a_2,c_2)
$$
and
$$
l_1 t_2=l_1' t_1.
$$
Since
\begin{alignat*}{2}
(b_1,d_1)&=l_1(a_1,c_1), \quad & (b_1',d_1')&=l_1'(a_1',c_1'),\\
(b_2,d_2)&=l_1^{-1}(a_2,c_2), & (b_2',d_2')&=l_1^{\prime -1}(a_2',c_2'),
\end{alignat*}
it follows that
$$
(b_2',d_2')=t_1(b_2,d_2), \; (b_1',d_1')=t_2(b_1,d_1).
$$
Thus
\begin{align*}
(a_2,b_1,c_2,d_1)&=l_1(b_2,a_1,d_2,c_1),\\
(a_2',b_1',c_2',d_1')&=l_1'(b_2',a_1',d_2',c_1'),\\
(a_1',b_2',c_1',d_2')&=t_1(a_1,b_2,c_1,d_2),\\
(a_2',b_1',c_2',d_1')&=t_2(a_2,b_1,c_2,d_1).
\end{align*}
(iv) of the theorem holds.

This settles (4-I.1.2).

\medskip
Case (4-I.2) $h_{23}=0$. Then  $a_1c_2-a_2c_1=0$.
Write $(a_1,a_2)=m(c_1,c_2)$ with $m\in F$.
Since $(b_1,d_1)=l_1(a_1,c_1)$, $(b_2,d_2)=l_2(a_2,c_2)$,
we have  $(b_1,b_2)=m(d_1,d_2)$.  So
$(a,b)=m(c,d)$.
Then
$(a',b')=m(c',d')$ as well.
(ii) of the theorem holds.

This settles (4-I.2).

Thus in Case (4-I) the disjunction of (i), (ii), (iv) of the theorem holds.

Case (4-I') is treated similarly.

Before going into the remaining cases we observe that 
when  $d=0$, we have
\begin{align*}
H=\begin{pmatrix}
a_1/c_1 & 0 & 0 & b_2/c_1\\ 
0 & a_2/c_2 & b_1/c_2 & 0\\
0 & b_2/c_1 & a_1/c_1 & b_1/c_1\\ 
b_1/c_2 & 0 & b_2/c_2 & a_2/c_2
\end{pmatrix},
\tag{47}
\end{align*}
$$
\det M_{c,d}=c_1^2c_2^2,
$$
and that when $a=0$ and  $c=0$, we have
\begin{align*}  
H=\begin{pmatrix}
b_1/d_1 & 0 & -(b_1 d_2 -b_2 d_1)/d_1^2 & 0\\
0 & b_2/d_2 & 0 & (b_1 d_2 -b_2 d_1)/d_2^2\\
0 & 0 & b_1/d_1 & 0\\ 
0 & 0 & 0 & b_2/d_2
\end{pmatrix},
\tag{48}
\end{align*}  
$$
\det M_{c,d}=d_1^2d_2^2.
$$

Case (4-II-II') $d_1d_2\ne 0$, $c_1=c_2=0$,
$d_1'd_2'\ne 0$, $c_1'=c_2'=0$.
By Lemma 4.2
we have
$a_1d_1-b_1c_1=0$, $a_2d_2-b_2c_2=0$, 
hence
$a_1=0$, $a_2=0$.
Similarly
$a_1'=0$, $a_2'=0$.
Thus $a=c=0$, $a'=c'=0$.
We have by (48)
\begin{align*}
h_{11}&=\frac{b_1}{d_1},\\
h_{22}&=\frac{b_2}{d_2},\\
h_{13}&=(\frac{b_2}{d_2}-\frac{b_1}{d_1})\frac{d_2}{d_1}.
\end{align*}
Hence
\begin{align*}
\frac{b_1}{d_1}&=\frac{b_1'}{d_1'},\\
\frac{b_2}{d_2}&=\frac{b_2'}{d_2'},\\
(\frac{b_2}{d_2}-\frac{b_1}{d_1})\frac{d_2}{d_1}&=
(\frac{b_2'}{d_2'}-\frac{b_1'}{d_1'})\frac{d_2'}{d_1'}.
\end{align*}
Case where $h_{13}\ne 0$.
Then
$$
\frac{b_2}{d_2}-\frac{b_1}{d_1}=\frac{b_2'}{d_2'}-\frac{b_1'}{d_1'}\ne 0,
$$
hence
$$
\frac{d_2}{d_1}=\frac{d_2'}{d_1'}.
$$
So
$(b_1,b_2,d_1,d_2)$ is proportional to $(b_1',b_2',d_1',d_2')$.
Since $a=0$, $c=0$, $a'=0$, $c'=0$, we see that
$(a,b,c,d)$ is proportional to $(a',b',c',d')$.
Thus (i) of the theorem holds.

Case where $h_{13}=0$.
Then
$$
\frac{b_2}{d_2}-\frac{b_1}{d_1}=\frac{b_2'}{d_2'}-\frac{b_1'}{d_1'}=0.
$$
Put
$$
\frac{b_1}{d_1}=\frac{b_2}{d_2}=
\frac{b_1'}{d_1'}=\frac{b_2'}{d_2'}=m.
$$
Then
$(b_1,b_2)=m(d_1,d_2)$, $(b_1',b_2')=m(d_1',d_2')$, so
$b=m d$, $b'=m d'$.
Since $a=c=0$ and $a'=c'=0$, we have
$(a,b)=m(c,d)$, $(a',b')=m(c',d')$.
Thus (ii) of the theorem holds.

This settles (4-II-II').

\medskip

Case (4-III-III') $c_1c_2\ne 0$,  $d_1=d_2=0$,
$c_1'c_2'\ne 0$, $d_1'=d_2'=0$.
The matrices $H$ and $H'$ are of the form (47). The assumption $H=H'$ gives
\begin{alignat*}{3}
\frac{a_1}{c_1}&=\frac{a_1'}{c_1'}, \quad&
\frac{a_2}{c_2}&=\frac{a_2'}{c_2'}, \quad&
\frac{b_2}{c_2}&=\frac{b_2'}{c_2'},\\
\frac{b_1}{c_2}&=\frac{b_1'}{c_2'},&
\frac{b_2}{c_1}&=\frac{b_2'}{c_1'},&
\frac{b_1}{c_1}&=\frac{b_1'}{c_1'}.
\end{alignat*}
Write $c_1'=t_1c_1$, $c_2'=t_2c_2$ with $t_1, t_2\in F^{\times}$.
Then $a_1'=t_1a_1$, $a_2'=t_2a_2$, and
$b_1'=t_1b_1=t_2b_1$, 
$b_2'=t_1b_2=t_2b_2$.

Case where $b_1\ne 0$. Then $t_1=t_2$, so
$(a',b',c')=t_1(a,b,c)$, 
$(a',b',c',d')=t_1(a,b,c,d)$.
(i) of the theorem holds.

Case where $b_1=0$. Then $b_2=0$ by condition ($\star$). Then $b_1'=0$, $b_2'=0$.
Thus
$b=0$, $b'=0$, $d=0$, $d'=0$. We have
$(a_1',c_1')=t_1(a_1,c_1)$, $(a_2',c_2')=t_2(a_2,c_2)$.
Thus (iii) of the theorem holds.

This settles (4-III-III').

\medskip
Case (4-II-III') 
$d_1d_2\ne 0$, $c_1=c_2=0$,
$c_1'c_2'\ne 0$, $d_1'=d_2'=0$.
Then, as was shown in Case (4-II-II'), we have $a_1=0$, $a_2=0$.
The matrix $H$ is of the form (48) and $H'$ of the form (47). The assumption $H=H'$ then gives  
\begin{align*}
h_{41}&=0=\frac{b_1'}{c_2'}, \displaybreak[0]\\
h_{32}&=0=\frac{b_2'}{c_1'}, \displaybreak[0]\\
h_{13}&=(\frac{b_2}{d_2}-\frac{b_1}{d_1})\frac{d_2}{d_1}=0, \displaybreak[0]\\
h_{11}&=\frac{b_1}{d_1}=\frac{a_1'}{c_1'}, \displaybreak[0]\\
h_{22}&=\frac{b_2}{d_2}=\frac{a_2'}{c_2'}.
\end{align*}
Hence
$$
b_1'=0,\; b_2'=0, \; \frac{b_1}{d_1}-\frac{b_2}{d_2}=0.
$$
So
$$
\frac{b_1}{d_1}=\frac{b_2}{d_2}=\frac{a_1'}{c_1'}=\frac{a_2'}{c_2'}.
$$
Call this value $m$.
Then
$b=m d$, $a'=m c'$.
Since $a=c=0$ and $b'=d'=0$, we have
$(a,b)=m(c,d)$, 
$(a',b')=m(c',d')$.
Thus (ii) of the theorem holds.

This settles (4-II-III').

\medskip
Case (4-III-II') $c_1c_2\ne 0$, $d_1=d_2=0$, $d_1'd_2'\ne 0$, $c_1'=c_2'=0$.
By exchange $(a,b,c,d)\leftrightarrow (a',b',c',d')$
we  have the same conclusion as above.

Thus in each of Cases (4-II-II'), (4-II-III'), (4-III-II'), (4-III-III')  the disjunction of  (i), (ii), (iii) of the theorem holds.

This settles Case (4).

With all cases examined, the proof of the theorem is complete.

\end{proof}

\bigskip
We shall rewrite the equations of (iv) into a more symmetrical form
under an additional assumption.

\begin{thm}
Let $(a,b), (c,d), (a',b'), (c',d')\in A$.
Assume that the matrices $M_{a,b}$, $M_{c,d}$, $M_{a',b'}$, $M_{c',d'}$ are invertible.
Assume that each of $a,b,\dots,c',d'$ satisfies ($\star$).
Then  $M_{c,d}^{-1}M_{a,b}= M_{c',d'}^{-1}M_{a',b'}$ if and only if the one of the following holds:

(i) $(a',b',c',d')=t(a,b,c,d)$ for some $t\in F^{\times}$.

(ii) $(a,b,a',b')=m(c,d,c',d')$ for some $m\in F^{\times}$.

(iii) $b=b'=0$, $d=d'=0$, $(a_1',c_1')=t_1(a_1,c_1)$, $(a_2',c_2')=t_2(a_2,c_2)$ for some $t_1,t_2\in F^{\times}$.

(iv) 
There exist $l,m,t_1,t_2\in F^{\times}$ such that
\begin{align*}
(a_2,b_2,a_2',b_2')&=l(b_1,a_1,b_1',a_1'), \\
(c_2,d_2,c_2',d_2')&=m(d_1,c_1,d_1',c_1') ,\\
(c_1,d_1,c_1',d_1')&=t_1(a_1,b_1,a_1',b_1'), \\
(c_2,d_2,c_2',d_2')&=t_2(a_2,b_2,a_2',b_2'). 
\end{align*}
(These equations imply $lt_2=mt_1$.)

\end{thm}

For the proof it is enough to show that  
 (iv) of Theorem 4.1 and (iv) of Theorem 4.5 are equivalent
under the present assumption.
We shall show this below.

Suppose that each of $a,b,\dots,c',d'\in F^2$ satisfies ($\star$).
The two statements in question:

(X) There exist $l,l',t_1,t_2\in F^{\times}$ such that
\begin{align*}
(a_2,b_1,c_2,d_1)&=l(b_2,a_1,d_2,c_1), \\
(a_2',b_1',c_2',d_1')&=l'(b_2',a_1',d_2',c_1'), \\
(a_1',b_2',c_1',d_2')&=t_1(a_1,b_2,c_1,d_2), \\
(a_2',b_1',c_2',d_1')&=t_2(a_2,b_1,c_2,d_1). 
\end{align*}

(Y) There exist $l,m,t_1,t_2\in F^{\times}$ such that
\begin{align*}
(a_2,b_2,a_2',b_2')&=l(b_1,a_1,b_1',a_1'), \\
(c_2,d_2,c_2',d_2')&=m(d_1,c_1,d_1',c_1'), \\
(c_1,d_1,c_1',d_1')&=t_1(a_1,b_1,a_1',b_1'), \\
(c_2,d_2,c_2',d_2')&=t_2(a_2,b_2,a_2',b_2'). 
\end{align*}

\begin{prop} Suppose that $M_{a,b}$, $M_{c,d}$, $M_{a',b'}$, $M_{c',d'}$ are invertible. Then
either of (X) and (Y) implies that the all components of $a,b,\dots,c',d'$ are  nonzero.
\end{prop}

\begin{proof}
Assume (X).
By the first equation of (X), if $a_1=0$, then $b_1=0$, so $a=0$, $b=0$ by ($\star$), contradicting the invertibility of  $M_{a,b}$. Thus $a_1\ne 0$, so $a_2\ne 0$.
Similarly $b_i\ne 0$, $c_i\ne 0$, etc.

Argue similarly for (Y).
\end{proof}

\begin{prop}
Suppose that the components of $a,b,\dots,c',d'$ are all nonzero.
Then (X) is equivalent to  (Y).
\end{prop}

\begin{proof}
Assume (X).
\begin{align*}
(a_2,b_1,c_2,d_1)&=l(b_2,a_1,d_2,c_1), \tag1\\
(a_2',b_1',c_2',d_1')&=l'(b_2',a_1',d_2',c_1'), \tag2\\
(a_1',b_2',c_1',d_2')&=t_1(a_1,b_2,c_1,d_2), \tag3\\
(a_2',b_1',c_2',d_1')&=t_2(a_2,b_1,c_2,d_1) \tag4
\end{align*}
with $l,l',t_1,t_2\in F^{\times}$.
By (1)
we can write
\begin{align*}
(a_2,b_2)&=\lambda(b_1,a_1),\tag5\\
(c_2,d_2)&=\mu(d_1,c_1)\tag6
\end{align*}
with $\lambda, \mu\in F^{\times}$.
By (2) we have similarly
\begin{align*}
(a_2',b_2')&=\lambda'(b_1',a_1'),\tag7\\
(c_2',d_2')&=\mu'(d_1',c_1')\tag8
\end{align*}
with $\lambda',\mu'\in F^{\times}$.

(3) says
$(a_1',b_2')=t_1(a_1,b_2)$, which by  (5) and (7)  becomes
$(a_1',\lambda' a_1')=t_1(a_1, \lambda a_1)$.
Since $a_1, a_1'\ne 0$, this implies 
$\lambda'=\lambda$.
Similarly
$\mu'=\mu$.

(5) and (7) are combined into
$$
(a_2,b_2,a_2',b_2')=\lambda(b_1,a_1,b_1',a_1'),
$$
and (6) and (8)  into
$$
(c_2,d_2,c_2',d_2')=\mu(d_1,c_1,d_1',c_1').
$$

By (3)
we can write
\begin{align*}
(c_1,c_1')&=\sigma_1(a_1,a_1'),\tag9\\
(d_2,d_2')&=\tau_2(b_2,b_2')\tag{12}
\end{align*}
with $\sigma_1, \tau_2\in F^{\times}$.
By (4)  we have similarly  
\begin{align*}
(c_2,c_2')&=\sigma_2(a_2,a_2'),\tag{11}\\
(d_1,d_1')&=\tau_1(b_1,b_1')\tag{10}
\end{align*}
with $\sigma_2,\tau_1\in F^{\times}$.

By (1)
$$
\vmatrix a_2 & d_1 \\
b_2 & c_1
\endvmatrix=0.
$$
By (5), (9), (10) this becomes
$$
\lambda\vmatrix b_1 & \tau_1 b_1 \\ 
a_1 & \sigma_1 a_1 \endvmatrix=0.
$$
Hence
$\sigma_1=\tau_1$.
Then (9) and (10) are combined into
$$
(c_1,d_1,c_1',d_1')=\tau_1(a_1,b_1,a_1',b_1').
$$
Similarly $\sigma_2=\tau_2$ and
$$
(c_2,d_2,c_2',d_2')=\tau_2(a_2,b_2,a_2',b_2').
$$

We obtain
\begin{align*}
(a_2,b_2,a_2',b_2')&=\lambda(b_1,a_1,b_1',a_1'),\\
(c_2,d_2,c_2',d_2')&=\mu(d_1,c_1,d_1',c_1'),\\
(c_1,d_1,c_1',d_1')&=\tau_1(a_1,b_1,a_1',b_1'),\\
(c_2,d_2,c_2',d_2')&=\tau_2(a_2,b_2,a_2',b_2').
\end{align*}
Thus (Y) holds.
This proves that (X) implies (Y).

The replacement of variables
$$
b_1\leftrightarrow b_2,\;
c_i\leftrightarrow a_i',\;
d_1\leftrightarrow b_2', \;
d_2\leftrightarrow b_1', \;
d_1'\leftrightarrow d_2',\;
l'\leftrightarrow m 
$$
with $a_i,c_i', l,t_i$ unchanged
interchanges  the equations of (X) and  the equations of (Y).

Therefore the statement $(Y)\implies (X)$ is true as well as $(X)\implies (Y)$.

This proves the proposition.
 
\end{proof}

By Propositions 4.6 and 4.7 we obtain Theorem 4.5 from Theorem 4.4.

\section{The non-split algebra of type V}

Here we prove Theorem 2.1, the main result for the algebra of type V.

Let $E/F$ be a quadratic extension with nontrivial automorphism $\sigma$.
Let $A$ be the algebra of type V associated with  $E/F$. 
As in Section 3 take a splitting field $K$ of the extension $E/F$ with injections $i_1,i_2\colon E\to K$.
We have a $K$-algebra isomorphism $\iota\colon K\otimes E\to K^2$ taking $1\otimes a\mapsto
(i_1(a), i_2(a))$.
Let $\tau$ be the automorphism of $K^2$ taking $(x_1,x_2)\mapsto (x_2,x_1)$.
The isomorphism $\iota$ transforms $1\otimes \sigma$ into $\tau$.

Let $\tilde A$ be the split $K$-algebra of type V associated with the extension $K^2/K$. 
We have a $K$-algebra isomorphism
$K\otimes A\to \tilde A$ taking 
$$
1\otimes (a,b) \mapsto (\iota(1\otimes a), \iota(1\otimes b))
=((i_1(a), i_2(a)), (i_1(b), i_2(b))).
$$

Let $(x,y), (x',y')\in A^2$.
Suppose that $R_x$, $R_y$, $R_{x'}$, $R_{y'}$ are invertible.
We know (Proposition 1.1)
$$
A(x,y)=A(x',y')\iff R_xR_y^{-1}=R_{x'}R_{y'}^{-1}.
$$
Write $x=(a,b), y=(c,d), x'=(a',b'),y'=(c',d')$ with $a,b,\dots,c',d'\in E$.
Write also
$\tilde a=(i_1(a),i_2(a))$, $\tilde b=(i_1(b), i_2(b)) \in K^2$, 
$\tilde x=(\tilde a,\tilde b)\in \tilde A$, etc.
Each of $\tilde a, \tilde b, \dots,\tilde c',\tilde d'$ satisfies condition ($\star$) of Section 4.
$R_{\tilde x}$, $R_{\tilde y}$, $R_{\tilde x'}$, $R_{\tilde y'}$ are invertible.
Clearly
$$
R_xR_y^{-1}=R_{x'}R_{y'}^{-1}
\iff
R_{\tilde x} R_{\tilde y}^{-1}=R_{\tilde x'}R_{\tilde y'}^{-1}.
$$
Let $M_{\tilde a, \tilde b}$ be the matrix defined in Section 3.
Then
$$
R_{\tilde x} R_{\tilde y}^{-1}=R_{\tilde x'}R_{\tilde y'}^{-1}
\iff
M_{\tilde c, \tilde d}^{-1}M_{\tilde a, \tilde b}=
 M_{\tilde c', \tilde d'}^{-1}M_{\tilde a',\tilde b'}.
$$

By Theorem 4.5 the last equality holds if and only if one of the following  holds.

(i) $(\tilde a', \tilde b', \tilde c', \tilde d')=t(\tilde a, \tilde b, \tilde c, \tilde d)$ for some $t\in K^{\times}$.

(ii) $(\tilde a, \tilde b,\tilde a', \tilde b')=m(\tilde c, \tilde d, \tilde c', \tilde d')$ for some $m\in K^{\times}$.

(iii) $\tilde b=\tilde b'=0$, $\tilde d=\tilde d'=0$, 
$(i_1(a'), i_1(c'))=t_1(i_1(a), i_1(c))$, $(i_2(a'), i_2(c'))=t_2(i_2(a), i_2(c))$ for some $t_1,t_2\in K^{\times}$.

(iv)
 There exist $l,m,t_1,t_2\in K^{\times}$ such that
\begin{align*}
(i_2(a),i_2(b), i_2(a'),i_2(b'))&=l(i_1(b),i_1(a),i_1(b'),i_1(a')), \\
(i_2(c),i_2(d), i_2(c'),i_2(d'))&=m(i_1(d),i_1(c),i_1(d'),i_1(c')), \\
(i_1(c),i_1(d),i_1(c'),i_1(d'))&=t_1(i_1(a),i_1(b),i_1(a'),i_1(b')), \\
(i_2(c),i_2(d),i_2(c'),i_2(d'))&=t_2(i_2(a),i_2(b),i_2(a'),i_2(b')), \\
lt_2&=mt_1.
\end{align*}

Let us rewrite these equations into equations in $E$.

(i) 
$(\tilde a', \tilde b', \tilde c', \tilde d')=t(\tilde a, \tilde b, \tilde c, \tilde d)$. The equalities
$\tilde a'=t\tilde a$, $\tilde b'=t \tilde b$ mean that
$i_1(a')=t i_1(a)$, $i_2(a')=t i_2(a)$,  $i_1(b')=t i_1(b)$,  $i_2(b')=t i_2(b)$.
Since $x=(a,b)\ne 0$, we have $a\ne 0$ or $b\ne 0$.
When $a\ne 0$, we have
$t=i_1(a'/a)=i_2(a'/a)$,
hence $a'/a\in F$, $t\in F^{\times}$.
Similarly, when $b\ne 0$, we have $t\in F^{\times}$.
Thus
$(a',b',c',d')=t(a,b,c,d)$.
Namely $x'=tx$, $y'=ty$.

(ii)
$(\tilde a, \tilde b,\tilde a', \tilde b')=m(\tilde c, \tilde d, \tilde c', \tilde d')$. The equalities
$\tilde a=m \tilde c$, $\tilde b=m \tilde d$ mean that
$i_1(a)=m i_1(c)$, $i_2(a)= m i_2(c)$,  $i_1(b)=m i_1(d)$,  $i_2(b)=m i_2(d)$.
Since $y=(c,d)\ne 0$, we have $c\ne 0$ or $d\ne 0$.
When $c\ne 0$,
$m=i_1(a/c)=i_2(a/c)$,
hence
$a/c\in F$ and $m\in F^{\times}$.
Similarly when $d\ne 0$, we have $m\in F^{\times}$.
In any case $m\in F^{\times}$ and
$(a,b,a',b')=m(c,d,c',d')$.
Namely
$x=m y$, $x'=my'$

(iii)
$\tilde b=\tilde b'=0$, $\tilde d=\tilde d'=0$ and
$$
(i_1(a'), i_1(c'))=t_1(i_1(a), i_1(c)),\; (i_2(a'), i_2(c'))=t_2(i_2(a), i_2(c)).
$$ 
Then $b=b'=0$, $d=d'=0$, 
so $a,c,a',c'\ne 0$. Then
$$
t_1=i_1(a'/a)=i_1(c'/c),\; t_2=i_2(a'/a)=i_2(c'/c).
$$
Write $t_1=i_1(t_0)$ with $t_0\in E^{\times}$. Then
$t_0=a'/a=c'/c$.
Thus
$(a',c')=t_0(a,c)$.

(iv)
\begin{align*}
(i_2(a),i_2(b), i_2(a'),i_2(b'))&=l(i_1(b),i_1(a),i_1(b'),i_1(a')), \\
(i_2(c),i_2(d), i_2(c'),i_2(d'))&=m(i_1(d),i_1(c),i_1(d'),i_1(c')), \\
(i_1(c),i_1(d),i_1(c'),i_1(d'))&=t_1(i_1(a),i_1(b),i_1(a'),i_1(b')), \\
(i_2(c),i_2(d),i_2(c'),i_2(d'))&=t_2(i_2(a),i_2(b),i_2(a'),i_2(b')), \\
lt_2&=mt_1.
\end{align*}
Then $l,m,t_1,t_2\in i_1(E)=i_2(E)$.
Write $l=i_1(l_0)$, $m=i_1(m_0)$, $t_1=i_1(t_0)$ with $l_0,m_0,t_0\in E^{\times}$.
Then the third and fourth equations imply $t_2=i_2(t_0)$. 
The first three imply
\begin{align*}
(a^{\sigma}, b^{\sigma}, a^{\prime\sigma}, b^{\prime\sigma})&=l_0(b,a,b',a'),\\
(c^{\sigma}, d^{\sigma}, c^{\prime\sigma}, d^{\prime\sigma})&=m_0(d,c,d',c'),\\
(c,d,c',d')&=t_0(a,b,a',b')
\end{align*}
and the last  implies
$$
l_0t_0^{\sigma}=m_0t_0.
$$
Thus
\begin{alignat*}{2}
(a^{\sigma}, b^{\sigma})&=l_0(b,a), \quad& 
(a^{\prime\sigma}, b^{\prime\sigma})&=l_0(b',a'),\\
 (c^{\sigma}, d^{\sigma})&=m_0(d,c), \quad&
 (c^{\prime\sigma}, d^{\prime\sigma})&=m_0(d',c'),\\
(c,d)&=t_0(a,b),\quad&
(c',d')&=t_0(a',b').
\end{alignat*}

This finishes the proof of Theorem 2.1.

\section{The algebra of type W}

In this section we prove Theorem 2.2.

Let $E/F$ be a quadratic extension with nontrivial automorphism $\sigma$.
Let $\omega\in E-F$.
Let $A$ be the algebra of type W associated with $E/F$:
$A=E\oplus E$, 
$(a,b)(c,d)=(ac+\omega b^{\sigma}d, bc+a^{\sigma}d)$.
This is a division algebra,  four-dimensional with basis
$(1,0), (\omega,0), (0,1), (0,\omega)$.
Put $e=(1,0)$, $j=(0,1)$. $e$ is the identity element.

Put $A_0=E\oplus 0=\{(a,0)\mid a\in E\}$.
This is a subalgebra isomorphic to the field $E$.
The multiplication rule specializes to
\begin{align*}
(a,0)(c,d)&=(ac, a^{\sigma}d),\\
(a,b)(c,0)&=(ac,bc).
\end{align*}
The left and right multiplication by elements of $A_0$  is an associative action 
on  $A$, so that $A$ becomes a two-sided $A_0$-module.
And
$(a,b)=(a,0)+j(b,0)$,
$A=A_0+jA_0$.

Moreover one sees 
$$
t(xy)=(tx)y,\; (xt)y=x(ty),\; (xy)t=x(yt)
$$
for $x,y\in A$, $t\in A_0$ (\cite{Knu}).

A computation shows

\begin{lem}
For any $x\in A$ we have $(jx)j=j(xj)$ if and only if $x\in A_0$.
\end{lem}

\begin{prop}
If $0\ne t\in A_0$, then $A(tx,ty)=A(x,y)$ for any $x,y\in A$.

\end{prop}

\begin{proof}
By the middle associativity for $t\in A_0$
$$
a(tx,ty)=(a(tx),a(ty))=((at)x,(at)y)=(at)(x,y).
$$
Hence
$A(tx,ty)=(At)(x,y)=A(x,y)$
for $0\ne t\in A_0$.
\end{proof}

\begin{prop}
If $T\in GL_2(A_0)$, then
the linear transformation $(x,y)\mapsto (x,y)T$ of $A^2$ maps the space
$A(x,y)$ isomorphically to $A((x,y)T)$.
\end{prop}

\begin{proof}
Follows  from the right associativity $(ax)t=a(xt)$ for $t\in A_0$.
\end{proof}

We restate Theorem 2.2 through the isomorphism $E\cong A_0$.

\begin{thm}
Let $x,y,x',y'\in A$ be nonzero elements.
Then $A(x,y)=A(x',y')$ if and only if either of the following holds:

(i) $x'=\lambda x$, $y'=\lambda y$ for some $\lambda\in A_0^{\times}$.

(ii) $x=y\mu$, $x'=y'\mu$
 for some $\mu\in A_0^{\times}$.
 
\end{thm}

\begin{proof}
(i) is sufficient for the equality $A(x,y)=A(x',y')$ by Proposition 6.2.

Assume (ii):
$x=y\mu$, $x'=y\mu'$ with $\mu\in A_0^{\times}$.
Put
$$
T=\begin{pmatrix} 1 & 0 \\ -\mu & 1 \end{pmatrix},
$$
so that
$(x,y)T=(0,y)$, $(x',y')T=(0,y')$.
Since $A$ is a division algebra, we have $A(0,y)=0\oplus A=A(0,y')$.
By Proposition 6.3 it follows that $A(x,y)=A(x',y')$.

Let us prove the necessity.
Write $x=(a,b)$, $y=(c,d)$ and
put
$$
T=\begin{pmatrix} (a,0) & (c,0) \\
(b,0) & (d,0) \end{pmatrix}
\in M_2(A_0),
$$
so that
$(x,y)=(e,j)T$.

First consider the case where $x,y$ are right linearly independent over $A_0$.
Then $T^{-1}$ exists and
$(x,y)T^{-1}=(e,j)$.
Suppose $A(x,y)=A(x',y')$.
Put $(x',y')T^{-1}=(\tilde x',\tilde y')$.
Then, by Proposition 6.3,
$A(e,j)=A(\tilde x',\tilde y')$.
Since
$$
A(e,j)=\{(ae,aj)\mid a\in A\}=\{(a,aj)\mid a\in A\},
$$
the last equation  means that
$(a\tilde x')j=a\tilde y' $ for all $a\in A$.
Letting $a=e, j$, we have equations
$\tilde x'j=\tilde y'$, 
$(j\tilde x')j=j\tilde y'$, 
hence
$(j\tilde x')j=j(\tilde x' j)$.
By Lemma 6.1 this means   $\tilde x'\in A_0$.
Put $\lambda=\tilde x'$. Then
$(\tilde x',\tilde y')=(\lambda, \lambda j)=\lambda(e,j)$, so
$$
(x',y')=(\lambda(e,j))T=\lambda((e,j)T)=\lambda(x,y).
$$
Thus (i) holds.

Next consider the case where $x,y$ are right linearly dependent over $A_0$.
 Write $y=xt$ with $t\in A_0$.
For any $a\in A$
$$
a(x,y)=(ax,ay)=(ax,a(xt))=(ax,(ax)t),
$$
hence
$A(x,y)=\{(a,at)\mid a\in A\}$.
Therefore,
if $A(x,y)=A(x',y')$, then $y'=x't$, so (ii) holds.

\end{proof}

\end{document}